\documentclass[final]{amsart}
\usepackage[utf8]{inputenc}
\usepackage[T1]{fontenc}
\usepackage{lmodern}
\usepackage[francais]{babel}
\usepackage{mathrsfs}
\usepackage{amsmath,amssymb,amsfonts}
\usepackage[unicode,pdfborder={0 0 0}]{hyperref}
\usepackage[notcite,notref]{showkeys}
\usepackage[ps,dvips,arrow,matrix,tips,line]{xy}
{\setbox0\hbox{$ $}}\fontdimen16\textfont2=\fontdimen17\textfont2
\entrymodifiers={+!!<0pt,\the\fontdimen22\textfont2>}
\SelectTips{cm}{10}

\newtheorem{thm}{Théorème}[section]
\newtheorem{lem}[thm]{Lemme}
\newtheorem{cor}[thm]{Corollaire}
\newtheorem{prop}[thm]{Proposition}
\theoremstyle{definition}
\newtheorem{definition}[thm]{Définition}
\newtheorem*{definition*}{Définition}
\theoremstyle{remark}
\newtheorem{remarques}[thm]{Remarques}
\newtheorem{remarque}[thm]{Remarque}
\newtheorem{exemple}[thm]{Exemple}
\newtheorem{question}[thm]{Question}
\numberwithin{equation}{section}

\renewcommand{\emptyset}{\varnothing}
\makeatletter
\def\myrightarrow{{\setbox\z@\hbox{$\rightarrow$}\dimen0\ht\z@\multiply\dimen0 6\divide\dimen0 10\ht\z@\dimen0\box\z@}}
\def\myrightarrowfill@{\arrowfill@\relbar\relbar\myrightarrow}
\newcommand{\myxrightarrow}[2][]{\ext@arrow 0359\myrightarrowfill@{#1}{#2}}
\makeatother
\edef\textsection{\textsection\penalty10000\hskip3.4pt}
\newcommand{\longisoto}{\myxrightarrow{\,\,\sim\,\,}}
\newcommand{\red}{{\mathrm{red}}}
\newcommand{\nr}{{\mathrm{nr}}}
\renewcommand{\C}{{\mathbf C}}
\newcommand{\Q}{{\mathbf Q}}
\newcommand{\Qp}{{\mathbf Q}_p}
\newcommand{\Z}{{\mathbf Z}}
\newcommand{\Zl}{{\mathbf Z}_\ell}
\newcommand{\F}{{\mathbf F}}
\newcommand{\cd}{{\mathrm{cd}}}
\newcommand{\Spec}{{\mathrm{Spec}}}
\newcommand{\Ker}{{\mathrm{Ker}}}
\newcommand{\PP}{\mathrm{P}}
\renewcommand{\P}{\mathbf{P}}
\newcommand{\sO}{\mathscr{O}}
\newcommand{\sX}{\mathscr{X}}
\renewcommand{\leq}{\leqslant}
\renewcommand{\geq}{\geqslant}

\title[Sur une conjecture de Kato et Kuzumaki]{Sur une conjecture de Kato et Kuzumaki concernant les hypersurfaces de Fano}
\author{Olivier Wittenberg}
\address{D\'epartement de math\'ematiques et applications, \'Ecole normale sup\'erieure, 45~rue d'Ulm, 75230 Paris Cedex 05, France}
\email{wittenberg@dma.ens.fr}
\date{29 août 2013; révisé le 19 octobre 2014}

\begin{document}

\begin{abstract}
Nous montrons que le corps~$\Q_p$ et les corps de nombres totalement imaginaires
vérifient la propriété~$C_1^1$ conjecturée par Kato et Kuzumaki en~1986.
Autrement dit, si~$k$ est l'un de ces corps et $f \in k[x_0,\dots,x_n]$ est un polynôme homogène de degré $d \leq n$,
tout élément de~$k$ s'écrit comme produit de normes depuis des extensions finies de~$k$ dans lesquelles~$f$ admet un zéro non trivial.
Nous établissons aussi la
conjecture d'Ax sur les corps pseudo-algébriquement clos parfaits pour les
corps dont le groupe de Galois absolu est un pro-$p$\nobreakdash-groupe.
\end{abstract}

\maketitle

\section{Introduction}

La conjecture dont il est question dans le titre fut proposée en~1986 par Kato et Kuzumaki~\cite{katokuzumaki} et
porte sur les liens entre dimension cohomologique des corps,
$K$\nobreakdash-théorie algébrique et hypersurfaces projectives de petit degré.

Voici son énoncé.
Pour tout corps~$k$ et tout entier $q\geq 0$, notons $K_q(k)$ le $q$\nobreakdash-ème
groupe de $K$\nobreakdash-théorie de Milnor de~$k$.  Si~$X$ est un $k$\nobreakdash-schéma de type fini, notons $N_q(X/k)$ le sous-groupe de $K_q(k)$
engendré par les images des applications norme $N_{k(x)/k}:K_q(k(x)) \to K_q(k)$
lorsque~$x$ parcourt l'ensemble des points fermés de~$X$
(voir \cite[\textsection1.7]{katogeneralizationlcft} pour la construction des applications norme).
Rappelons quelques exemples classiques. Le groupe $N_0(X/k)$ est le sous-groupe de $K_0(k)=\Z$ engendré par l'indice de~$X$ sur~$k$.
Si~$X$ est une variété de Severi--Brauer, le groupe $N_1(X/k) \subset K_1(k)=k^*$ est l'image de la norme réduite $\mathrm{Nrd}:A^*\to k^*$, où~$A$ désigne
l'algèbre centrale simple correspondant à~$X$.  Si~$X$ est une quadrique projective,
le groupe~$N_1(X/k)$ contient~$k^{*2}$ et
le quotient coïncide avec l'image de la norme spinorielle associée à une forme quadratique définissant~$X$
(conséquence du principe de norme de Knebusch, cf.~\cite{knebusch}, \cite[Satz~A]{kneser}).
D'après un théorème de Merkurjev et Suslin, un corps parfait~$k$ est
de dimension cohomologique~$\leq 2$
si et seulement si la norme réduite
de toute algèbre centrale simple sur toute extension finie de~$k$ est surjective
(cf.~\cite[Chapitre~II, \textsection4.5]{serrecg}).

Soit $i \geq 0$ un entier.
Suivant~\cite{katokuzumaki}, on dit que le corps~$k$ vérifie la propriété~$C_i^q$ si pour toute extension finie $k'/k$, tout entier $n \geq 1$ et toute hypersurface
$X \subset \P^n_{k'}$ de degré~$d$
avec~$d^i \leq n$, l'égalité $K_q(k')=N_q(X/k')$ a lieu.
En ces termes, la conjecture avancée dans~\cite{katokuzumaki} est la suivante:
quels que soient les entiers $i,q\geq 0$
et le corps~$k$,
la propriété~$C_i^q$ vaut pour~$k$
si et seulement si~$k$ est de dimension~$\leq i+q$.
La notion de dimension employée ici coïncide avec la dimension cohomologique si~$k$ est parfait; sans hypothèse sur~$k$, la dimension majore la dimension cohomologique
(voir \cite{katokuzumaki} pour la définition générale).

Des exemples de corps de caractéristique~$0$ et de dimension cohomologique~$i$ ne vérifiant pas la propriété~$C_i^0$
furent donnés par Merkurjev~\cite{merkurjev} pour $i=2$ puis par Colliot-Thélène et Madore~\cite{ctmadore} pour~$i=1$.
Ainsi la conjecture de~\cite{katokuzumaki} est-elle trop optimiste en toute généralité.
Cependant, les contre-exemples connus reposent tous sur le procédé de
construction de grands corps par récurrence transfinie
développé par Merkurjev et Suslin (cf.~\cite[\textsection11.4]{merkurjevsuslin}).
La conjecture de Kato et Kuzumaki
pour les corps apparaissant naturellement en géométrie algébrique
et en arithmétique reste donc un problème ouvert.

L'un des corps les plus simples pour lesquels la conjecture n'est pas résolue est le corps~$\Qp$ des nombres $p$\nobreakdash-adiques,
qui est de dimension cohomologique~$2$ mais n'est pas~$C_2$ au sens d'Artin et Lang (cf.~\cite{terjanian}).  Il vérifie la propriété~$C_0^2$ d'après Bass et Tate
(cf.~\cite[Corollary~A.15]{milnor}).
C'est une question ouverte de savoir si~$\Q_p$ satisfait la propriété~$C_2^0$.
L'un des buts du présent article, atteint au~\textsection\ref{sec:corpspadiquesC11}, est de démontrer que~$\Q_p$
vérifie la propriété~$C_1^1$.
Autrement dit, pour tout corps $p$\nobreakdash-adique~$k$, tout entier $n\geq 1$ et toute hypersurface $X \subset \P^n_k$ de degré $d \leq n$,
tout élément de~$k^*$ s'écrit comme produit de normes depuis des extensions finies variables~$k'/k$ telles que $X(k')\neq\emptyset$.
Cette propriété concerne les hypersurfaces de Fano.
Nous montrons aussi que les corps de nombres totalement imaginaires sont des corps~$C_1^1$
et que les corps locaux supérieurs de dimension cohomologique~$d$ et de caractéristique résiduelle~$p$
vérifient la propriété~$C_1^{d-1}$ \og{}hors de~$p$\fg{}
(c'est-à-dire que
pour toute hypersurface $X \subset \P^n_k$ de degré~$\leq n$, le quotient $K_{d-1}(k)/N_{d-1}(X/k)$ est annulé par une puissance de~$p$).

Tous ces énoncés, ainsi que la propriété~$C_2^0$ pour les corps $p$\nobreakdash-adiques, avaient été établis
par Kato et Kuzumaki dans~\cite{katokuzumaki} sous la restriction que les hypersurfaces considérées
sont de degré \emph{premier} (et, dans le cas des corps de nombres totalement imaginaires, sous
la restriction supplémentaire qu'elles sont lisses).
Les arguments de~\cite{katokuzumaki} reposent sur des considérations élémentaires remarquables
concernant les polynômes homogènes et les extensions finies de corps locaux,
et sur le théorème de Chevalley--Warning.
La méthode employée ici est différente.
Notre point de départ est la définition d'une variante de la propriété~$C_1^q$ étendue à tous les schémas propres
sur le corps considéré:

\begin{definition*}
Soit $q\geq 0$ un entier.  Un corps~$k$ vérifie la propriété~$C_1^q$ \emph{forte} si pour toute extension finie~$k'/k$, tout $k'$\nobreakdash-schéma propre~$X$ et tout faisceau cohérent~$E$ sur~$X$,
la caractéristique d'Euler--Poincaré $\chi(X,E)=\sum (-1)^i \dim_{k'}H^i(X,E)$
annule
le groupe abélien $K_q(k')/N_q(X/k')$.
\end{definition*}

Cette définition en termes de caractéristiques d'Euler--Poincaré est directement inspirée de~\cite{elw}, où il est établi que le corps~$\C((t))$ vérifie la propriété~$C_1^0$ forte
et que l'extension non ramifiée maximale d'un corps $p$\nobreakdash-adique
vérifie la propriété~$C_1^0$ forte hors de~$p$.

Comme le quotient $K_q(k')/N_q(X/k')$ est annulé par le degré sur~$k'$ de tout point fermé de~$X$
(cf.~\cite[Remark~7.3.1]{gilleszamuely}),
la condition apparaissant dans la définition de la propriété~$C_1^q$ forte est triviale
pour les faisceaux cohérents~$E$ supportés en dimension~$0$.
Prenant à l'opposé pour~$E$ le faisceau structural~$\sO$,
on constate que
la propriété~$C_1^q$ forte entraîne la propriété~$C_1^q$
puisque
les hypersurfaces $X \subset \P^n$ de degré~$d \leq n$
vérifient $\chi(X,\sO)=1$.
Grâce à la souplesse laissée à la fois dans le choix de~$E$ et dans celui de~$X$,
la propriété~$C_1^q$ forte se prête bien mieux aux dévissages que la propriété~$C_1^q$.
C'est la flexibilité qui en résulte qui nous permettra de généraliser les résultats de~\cite{katokuzumaki}
portant sur~$C_1^q$.

Le plan du texte est le suivant.  Au~\textsection\ref{sec:devissage},
nous axiomatisons le principe de dévissage sous-jacent à la démonstration de~\cite[Theorem~3.1]{elw}.
Une première application en est donnée au~\textsection\ref{sec:ax}: nous y démontrons
la conjecture d'Ax
sur les corps pseudo-algébriquement clos parfaits dans le cas où le groupe de Galois absolu du corps considéré
est un pro-$p$\nobreakdash-groupe
(la conjecture était précédemment connue pour les corps de caractéristique~$0$, pour les corps contenant un corps
algébriquement clos et pour les corps de groupe de Galois absolu abélien).
Nous démontrons aussi au~\textsection\ref{sec:ax} que les corps finis vérifient la propriété~$C_1^0$ forte;
cet énoncé servira par la suite de substitut pour le théorème de Chevalley--Warning.
Le~\textsection\ref{sec:transition} est consacré à un théorème de transition pour la propriété~$C_1^q$ forte.
De ce théorème et de la propriété~$C_1^0$ forte pour les corps finis, il résulte que le corps~$\Q_p$ vérifie la propriété~$C_1^1$ forte hors de~$p$.
Le~\textsection\ref{sec:corpspadiquesC11} contient les raffinements géométriques nécessaires à la démonstration de la propriété~$C_1^1$ en~$p$ pour~$\Q_p$.
Au~\textsection\ref{sec:cdn} nous établissons, à l'aide des résultats des~\textsection\ref{sec:transition} et~\ref{sec:corpspadiquesC11}
et d'un théorème de Kato et Saito,
la propriété~$C_1^1$ pour les corps de nombres totalement imaginaires.  Enfin, le~\textsection\ref{sec:qrem} contient quelques compléments,
remarques
et questions ouvertes, y compris notamment une solution
au problème \cite[\textsection5, Problem~1]{katokuzumaki}, dans le cas $i=1$, pour les corps hilbertiens.

\medskip
\emph{Remerciements.}
Une partie des idées qui interviennent dans la preuve
de la conjecture~$C_1^1$ pour les corps $p$\nobreakdash-adiques
ont pour origine l'article~\cite{elw}, écrit en commun avec Hélène~Esnault et Marc~Levine
et qui concernait l'indice des variétés sur le corps des fractions d'un anneau de valuation discrète
hensélien à corps résiduel algébriquement clos.
Je leur suis reconnaissant pour les échanges que nous avons eus lors de cette collaboration.
Je remercie d'autre part
Hélène Esnault de ses encouragements amicaux et de nos discussions
sur la conjecture~$C_2^0$ pour les corps $p$\nobreakdash-adiques,
Philippe~Gille de ses réponses à mes questions sur les groupes de normes d'espaces homogènes de groupes linéaires,
Olivier Benoist d'utiles discussions concernant le~\textsection\ref{sec:exemple},
Dan~Abramovich de ses explications sur le lemme~\ref{lem:redsemistable}
et les rapporteurs de leur lecture très attentive.

\medskip
\emph{Conventions.}
Si~$X$ est un schéma de type fini sur un corps~$k$, l'\emph{indice} de~$X$ sur~$k$ est le pgcd des degrés des points fermés de~$X$.
La notation $K_q(k)$ désigne le $q$\nobreakdash-ème groupe de $K$\nobreakdash-théorie
de Milnor (cf.~\cite{milnordef}, \cite[Chapter~7]{gilleszamuely})
et $N_q(X/k) \subset K_q(k)$ le $q$\nobreakdash-ème \emph{groupe de normes} de~$X$ sur~$k$, c'est-à-dire le plus petit sous-groupe de~$K_q(k)$
contenant $N_{k(x)/k}(K_q(k(x)))$ pour tout point fermé~$x$ de~$X$.
Rappelons que $K_0(k)=\Z$ et $K_1(k)=k^*$.
Si~$X$ n'est pas vide, l'indice de~$X$ sur~$k$ est l'ordre du quotient $K_0(k)/N_0(X/k)$.

\section{Principe de dévissage}
\label{sec:devissage}

La proposition~\ref{prop:devissage} ci-dessous replace dans un cadre général les arguments de dévissage employés dans~\cite{elw}.
Nous nous en servirons à de nombreuses reprises dans les~\textsection\ref{sec:ax} à~\ref{sec:qrem}.

\begin{prop}
\label{prop:devissage}
Soit~$k$ un corps.
Soit~$\PP$ une propriété des $k$\nobreakdash-schémas propres.
Supposons donné, pour chaque $k$\nobreakdash-schéma propre~$X$, un entier $n_X \in \Z$.
Supposons les trois conditions suivantes satisfaites:
\begin{enumerate}
\item Pour tout morphisme de $k$\nobreakdash-schémas propres $Y \to X$, l'entier~$n_X$ divise~$n_Y$.
\item Pour tout $k$\nobreakdash-schéma propre~$X$ vérifiant~$\PP$,
l'entier~$n_X$ divise $\chi(X,\sO_X)$.
\item Pour tout $k$\nobreakdash-schéma propre et intègre~$X$, il existe un $k$\nobreakdash-schéma propre~$Y$ vérifiant~$\PP$
et un $k$\nobreakdash-morphisme $f:Y\to X$
tel que les entiers $\chi(Y_\eta,\sO_{Y_\eta})$ et~$n_X$ soient premiers entre eux, où~$Y_\eta$ désigne la fibre générique de~$f$.
\end{enumerate}
Alors pour tout $k$\nobreakdash-schéma propre~$X$ et tout faisceau cohérent~$E$ sur~$X$, l'entier~$n_X$ divise $\chi(X,E)$.
\end{prop}

\begin{proof}
Procédons par récurrence sur la dimension du schéma~$X$
apparaissant dans la
conclusion de la proposition.
La conclusion étant
satisfaite pour~$X$ vide, on peut la supposer
satisfaite par tout $k$\nobreakdash-schéma propre de dimension~$<\dim(X)$.
Par ailleurs,
comme le groupe de Grothendieck~$G_0(X)$ des faisceaux cohérents sur~$X$ est engendré par les classes des faisceaux~$\sO_Z$ lorsque~$Z$ parcourt
l'ensemble des sous-schémas fermés
intègres de~$X$
(cf.~\cite[\textsection8, Lemme~17]{BorelSerre},
\cite[Proposition~1.1]{BerthelotSGA6}),
il suffit de montrer que~$n_X$ divise $\chi(Z,\sO_Z)$ pour tout tel~$Z$.
Quitte à remplacer~$X$ par~$Z$,
on peut,
grâce à la propriété~(1) appliquée à l'inclusion de~$Z$ dans~$X$,
supposer que~$X$ est intègre et que $E=\sO_X$.
Soit alors $f:Y\to X$ donné par la propriété~(3).
Tout faisceau cohérent sur un schéma noethérien réduit étant génériquement libre,
il existe un ouvert dense $U \subset X$ tel que la restriction de $R^q f_* \sO_Y$ à~$U$ soit libre pour tout $q \geq 0$.
Notons $i:C \hookrightarrow X$ l'inclusion du sous-schéma fermé réduit $C=X \setminus U$.
D'après la suite exacte de localisation pour le groupe~$G_0(X)$,
il existe un faisceau cohérent virtuel~$F$ sur~$C$ tel que
\begin{align}
\sum_{q \geq 0} (-1)^q [R^q f_* \sO_Y] = \chi(Y_\eta,\sO_{Y_\eta})[\sO_X] + [i_*F]\rlap{\text{,}}
\end{align}
où les crochets désignent les classes dans~$G_0(X)$
(cf.~\cite[\textsection8, Proposition~7]{BorelSerre}).
Il~s'ensuit que
\begin{align}
\label{eq:chichichi}
\chi(Y,\sO_Y)=\chi(Y_\eta,\sO_{Y_\eta})\mkern3mu\chi(X,\sO_X)+\chi(C,F)
\end{align}
(\emph{loc.\ cit.}, \textsection5d).
L'entier~$n_Y$ divise $\chi(Y,\sO_Y)$ d'après~(2),
l'entier~$n_C$ divise~$\chi(C,F)$
par hypothèse de récurrence; enfin, l'entier~$n_X$ divise~$n_Y$ et~$n_C$ d'après~(1).
Comme~$n_X$ et~$\chi(Y_\eta,\sO_{Y_\eta})$ sont premiers entre eux,
on conclut ainsi de~\eqref{eq:chichichi} que~$n_X$ divise~$\chi(X,\sO_X)$.
\end{proof}

\begin{remarques}
\label{rq:devissage}
(i)
Dans toutes les situations où nous appliquerons la proposition~\ref{prop:devissage},
le morphisme $f:Y\to X$ de la propriété~(3) sera génériquement fini et dominant.
Pour un tel morphisme, on a $\chi(Y_\eta,\sO_{Y_\eta})=\deg(f)$.

(ii)
Si l'on sait seulement que les propriétés~(1) à~(3) sont satisfaites pour les schémas~$X$ de dimension~$\leq d$,
où~$d$ est un entier fixé,
la preuve ci-dessus assure tout de même la validité de
la conclusion de la proposition~\ref{prop:devissage} pour les $k$\nobreakdash-schémas propres~$X$ de dimension~$\leq d$.
\end{remarques}

\begin{exemple}
\label{exemple:devissage}
Soit~$R$ un anneau de valuation discrète hensélien excellent, de corps des fractions~$K$, de corps résiduel algébriquement clos.
Soit~$\ell$ un nombre premier inversible dans~$R$.  Si~$X$ est un $K$\nobreakdash-schéma propre,
notons~$n_X$ la plus grande puissance de~$\ell$ divisant l'indice de~$X$ sur~$K$ si~$X$ est non vide, ou~$0$ sinon,
et disons que~$X$ possède la propriété~$\PP$
s'il existe un $R$\nobreakdash-schéma régulier, propre et plat de fibre générique isomorphe à~$X$.
La propriété~(1) est évidente, la propriété~(3) est satisfaite d'après un
théorème de Gabber et de~Jong
\cite[Theorem~1.4]{IllusieGabber}
et la propriété~(2) est \cite[Theorem~2.1]{elw}.
La~proposition~\ref{prop:devissage} permet donc de retrouver \cite[Theorem~3.1]{elw}.
\end{exemple}

\section{Une application à la conjecture d'Ax}
\label{sec:ax}

Rappelons qu'un corps~$k$ est \emph{pseudo-algébriquement clos} si tout $k$\nobreakdash-schéma
de type fini géométriquement intègre possède un point rationnel (cf.~\cite{axpac}, \cite{friedjarden}).
D'après une conjecture d'Ax, tout corps pseudo-algébriquement clos parfait devrait être~$C_1$
au sens de~\cite{langqac}.  Cette conjecture est démontrée pour les corps
de caractéristique~$0$ (Kollár~\cite{kollarax}), pour les corps contenant un corps algébriquement clos (Denef--Jarden--Lewis~\cite{denefjardenlewis})
et pour les corps dont le groupe de Galois absolu est abélien (Ax~\cite[Theorem~D]{axpac}).
Nous allons voir
qu'une simple application du principe de dévissage du~\textsection\ref{sec:devissage}
permet d'établir la conjecture d'Ax pour les corps dont le groupe de Galois absolu est un pro-$p$\nobreakdash-groupe
(non nécessairement abélien).

\begin{thm}
\label{th:axconj}
Soit~$p$ un nombre premier.
Soit~$k$ un corps dont le groupe de Galois absolu soit un pro-$p$\nobreakdash-groupe.
Si $X \subset \P^n_k$ est une hypersurface de degré $d \leq n$,
il existe un fermé $W \subset X$ géométriquement irréductible sur~$k$.
\end{thm}

\begin{cor}
\label{cor:axconj}
Soit~$p$ un nombre premier.
Tout corps pseudo-algébriquement clos parfait dont le groupe de Galois absolu est un pro-$p$\nobreakdash-groupe est un corps~$C_1$.
\end{cor}

Nous allons déduire le théorème~\ref{th:axconj} de la proposition suivante.
Si~$k$ est un corps et~$W$ un $k$\nobreakdash-schéma intègre, notons $k_W$ la fermeture algébrique de~$k$ dans~$k(W)$.
Si~$X$ est un $k$\nobreakdash-schéma de type fini, notons~$i_X$ le pgcd des degrés~$[k_W:k]$
lorsque~$W$ parcourt l'ensemble des sous-schémas fermés intègres de~$X$.

\begin{prop}
\label{prop:indiceirred}
Pour tout corps~$k$, tout schéma~$X$ propre sur~$k$ et tout faisceau cohérent~$E$ sur~$X$,
l'entier~$i_X$ divise $\chi(X,E)$.
\end{prop}

\begin{proof}
Fixons le corps~$k$.
Si~$X$ est un $k$\nobreakdash-schéma propre,
posons $n_X=i_X$ et notons~$\PP$ la propriété, pour~$X$, d'être irréductible et normal.
La propriété~(1) de la proposition~\ref{prop:devissage} est évidente.
La propriété~(3) l'est aussi: prendre pour~$f$ le morphisme de normalisation.
Considérons la propriété~(2).
Soit~$X$ un $k$\nobreakdash-schéma propre, irréductible et normal.
Comme le $k$\nobreakdash-schéma~$X$ est normal, il possède naturellement une structure de $k_X$\nobreakdash-schéma.
La formule
\begin{align}
\label{eq:multiplicativitedim}
\dim_k(V)=[k_X:k]\dim_{k_X}(V)\rlap{\text{,}}
\end{align}
valable pour tout $k_X$\nobreakdash-espace vectoriel~$V$ de dimension finie,
entraîne que $[k_X:k]$ divise $\chi(X,\sO_X)$.
Par conséquent~$i_X$ divise bien $\chi(X,\sO_X)$.
La proposition~\ref{prop:devissage} permet de conclure.
\end{proof}

\begin{proof}[Démonstration du théorème~\ref{th:axconj}]
Soit~$k$ un corps.
Soit $X \subset \P^n_k$ une hypersurface de degré $d \leq n$.
Une telle hypersurface vérifie $\chi(X,\sO_X)=1$.
Il existe donc,
d'après la proposition~\ref{prop:indiceirred},
un fermé irréductible $W \subset X$ tel que le degré de l'extension finie~$k_W/k$ soit premier à~$p$.
Si le groupe de Galois absolu de~$k$ est un pro-$p$\nobreakdash-groupe,
l'extension~$k_W/k$ est alors purement inséparable;
de façon équivalente, le fermé~$W$
est géométriquement irréductible sur~$k$
(cf.~\cite[Proposition~4.5.9]{ega42}).
\end{proof}

\begin{cor}
Soit~$k$ un corps pseudo-algébriquement clos parfait.
Toute hypersurface $X \subset \P^n_k$ de degré $d\leq n$ est d'indice~$1$.
\end{cor}

\begin{proof}
Soit~$p$ un nombre premier.
Notons~$\bar k$ une clôture algébrique de~$k$.
D'après le corollaire~\ref{cor:axconj} appliqué
au sous-corps de~$\bar k$ fixe par un $p$\nobreakdash-groupe
de Sylow de $\mathrm{Gal}(\bar k/k)$, l'hypersurface~$X$ acquiert un point rationnel après
une extension des scalaires de degré premier à~$p$.
Par conséquent, l'indice de~$X$ est premier à~$p$.
\end{proof}

Une autre conséquence de la proposition~\ref{prop:indiceirred} est la validité
de
l'énoncé \cite[Theorem~3.1]{elw} non seulement pour le corps~$\C((t))$ mais aussi pour les corps finis.

\begin{cor}
\label{cor:kfini}
Soit~$k$ un corps fini.  Pour tout $k$\nobreakdash-schéma propre~$X$ et tout faisceau cohérent~$E$ sur~$X$,
l'indice de~$X$ sur~$k$ divise $\chi(X,E)$.
\end{cor}

\begin{proof}
Lorsque~$k$ est un corps fini, l'entier~$i_X$ coïncide avec l'indice de~$X$ sur~$k$.
En effet, tout sous-schéma fermé intègre~$W$ de~$X$
contient un ouvert dense géométriquement irréductible sur~$k_W$.
Un tel ouvert est d'indice~$1$ sur le corps fini~$k_W$
d'après
Lang--Weil--Nisnevi\v{c} (cf.~\cite{langweil}, \cite{nisnevic});
l'indice de~$X$ sur~$k$ divise donc~$[k_W:k]$.
\end{proof}

Nous nous servirons du corollaire~\ref{cor:kfini} aux \textsection\ref{sec:transition} à~\ref{sec:cdn}.

\section{Théorème de transition pour la propriété \texorpdfstring{$C_1^q$}{C\_1\textasciicircum{}q} forte}
\label{sec:transition}

\begin{definition}
\label{def:c1forte}
Soit $q\geq 0$ un entier.  Soit~$\ell$ un nombre premier.
Un corps~$k$ vérifie la propriété~$C_1^q$ \emph{forte} (resp.~la propriété~$C_1^q$ \emph{forte en~$\ell$}, la propriété~$C_1^q$ \emph{forte hors de~$\ell$}) si pour toute extension finie $k'/k$,
tout $k'$\nobreakdash-schéma propre~$X$ et tout faisceau cohérent~$E$ sur~$X$,
la caractéristique d'Euler--Poincaré $\chi(X,E)=\sum (-1)^i \dim_{k'} H^i(X,E)$
annule le groupe abélien $K_q(k')/N_q(X/k')$ (resp.~annule son sous-groupe de torsion $\ell$\nobreakdash-primaire, son sous-groupe de torsion première à~$\ell$).
\end{definition}

Kato et Kuzumaki~\cite[Theorem~1]{katokuzumaki} ont démontré que si~$R$ est un anneau de valuation discrète hensélien excellent dont le corps résiduel~$k$
vérifie la propriété~$C_0^q$ pour un entier $q\geq 1$ et si~$d$ est un nombre \emph{premier} inversible dans~$R$,
alors la propriété $C_1^{q-1}(d)$ pour~$k$ implique la propriété~$C_1^q(d)$
pour le corps des fractions de~$R$, où~$C_1^{q-1}(d)$ et~$C_1^q(d)$ désignent les propriétés~$C_1^{q-1}$ et~$C_1^q$ restreintes aux hypersurfaces de degré~$d$.
C'est une question ouverte de savoir si le même énoncé reste vrai sans l'hypothèse que~$d$ est premier.
Nous montrons dans ce paragraphe que la situation est plus favorable
pour la propriété~$C_1^q$ \emph{forte} hors de la caractéristique de~$k$: elle vérifie
un théorème de transition sans restriction sur les schémas considérés.

\begin{thm}
\label{th:transition}
Soit~$R$ un anneau de valuation discrète hensélien excellent de corps des fractions~$K$, de corps résiduel~$k$.
Soit~$\ell$ un nombre premier inversible dans~$R$.
Soit $q \geq 1$ un entier.  Si~$k$ vérifie la propriété $C_1^{q-1}$ forte en~$\ell$
et la propriété $C_0^q$ en~$\ell$, alors~$K$ vérifie
la propriété $C_1^q$ forte en~$\ell$.
\end{thm}

\begin{remarques}
\label{rem:transition}
(i)
Par définition, le corps~$k$ vérifie la propriété~$C_0^q$ en~$\ell$ si et seulement si pour toute extension finie~$k'/k$
et toute extension finie~$k''/k'$, la $\ell$\nobreakdash-torsion du groupe $K_q(k')/N_{k''/k'}(K_q(k''))$ est nulle.
Selon la conjecture de Bloch--Kato, établie par Rost et Voevodsky, cette propriété est satisfaite,
pour~$\ell$ inversible dans~$k$,
si et seulement si $\cd_\ell(k)\leq q$ (cf.~\cite[Lemma~7]{katokuzumaki}).

(ii)
Indépendamment de ce résultat, Kato~\cite[\textsection3.3]{katogeneralizationlcft} avait démontré que si~$k$ vérifie la propriété~$C_0^q$ en~$\ell$, alors~$K$ vérifie la propriété~$C_0^{q+1}$ en~$\ell$.
Ainsi, les hypothèses du théorème~\ref{th:transition} se transmettent de~$k$ à~$K$:
si~$k$ vérifie les propriétés~$C_1^{q-1}$ forte en~$\ell$ et~$C_0^q$ en~$\ell$,
alors~$K$ vérifie les propriétés~$C_1^q$ forte en~$\ell$ et~$C_0^{q+1}$ en~$\ell$.
\end{remarques}

\begin{proof}[Démonstration du théorème~\ref{th:transition}]
Les hypothèses du théorème portant sur~$R$ et sur~$k$ restant satisfaites si l'on remplace~$K$ par une extension finie~$K'$ et~$R$ par sa fermeture intégrale dans~$K'$
(cf.~\cite[Chapitre~I, \textsection4, Proposition~9]{serrecorpslocaux}, \cite[Scholie~7.8.3~(vi)]{ega42},
\cite[Proposition~18.5.9~(ii)]{ega44}),
il suffit de démontrer que pour tout $K$\nobreakdash-schéma propre~$X$ et tout faisceau cohérent~$E$ sur~$X$,
le sous-groupe de torsion $\ell$\nobreakdash-primaire de $K_q(K)/N_q(X/K)$ est annulé par~$\chi(X,E)$.

Pour tout $K$\nobreakdash-schéma propre~$X$, notons~$n_X$ la plus grande puissance de~$\ell$
divisant l'exposant du groupe $K_q(K)/N_q(X/K)$ si~$X$ est non vide ou~$0$ si~$X$ est vide,
et convenons que~$X$ possède la propriété~$\PP$
s'il existe un $R$\nobreakdash-schéma irréductible, régulier, propre et plat de fibre générique isomorphe à~$X$.
Il suffit, pour établir le théorème, de vérifier les hypothèses~(1), (2), (3) de la proposition~\ref{prop:devissage}.
La~propriété~(1) est évidente.
Comme dans l'exemple~\ref{exemple:devissage}, la propriété~(3) est satisfaite
grâce au théorème de Gabber et de~Jong \cite[Theorem~1.4]{IllusieGabber},
applicable parce que~$\ell$ est inversible dans~$R$.
Reste la propriété~(2).
Pour la vérifier, fixons un $R$\nobreakdash-schéma~$\sX$ irréductible, régulier, propre et plat, posons $X =\sX\otimes_RK$ et montrons que~$n_X$ divise $\chi(X,\sO_X)$.

Rappelons que l'on dispose d'une application résidu $\partial:K_q(K) \to K_{q-1}(k)$. Celle-ci est surjective et son noyau est le sous-groupe
$U_q(K) \subset K_q(K)$ engendré par les symboles $\{u_1,\dots,u_q\}$ où les~$u_i$ sont des éléments de~$R^*$
(cf.~\cite[Proposition~7.1.7]{gilleszamuely}).  L'application~$\partial$ induit une suite exacte de groupes abéliens d'exposant fini
\begin{align}
\label{se:decoupe}
\xymatrix@C=2em{
0 \ar[r] & \displaystyle\frac{U_q(K)}{U_q(K) \cap N_q(X/K)} \ar[r] & \displaystyle\frac{K_q(K)}{N_q(X/K)} \ar[r]^(.46)\partial & \displaystyle\frac{K_{q-1}(k)}{\partial(N_q(X/K))} \ar[r] & 0\rlap{\text{.}}
}
\end{align}

Notons $Y = \sX \otimes_R k$ la fibre spéciale de~$\sX$.  La \emph{multiplicité} de~$Y$ est le plus grand entier~$m$ tel que~$Y$ soit divisible par~$m$
en tant que diviseur sur~$\sX$.

\begin{lem}
\label{lem:multipliciteU}
La multiplicité de~$Y$ annule le sous-groupe de torsion $\ell$\nobreakdash-primaire du groupe
abélien $U_q(K)/(U_q(K)\cap N_q(X/K))$.
\end{lem}

\begin{proof}
Notons~$\mathfrak{m}$ l'idéal maximal de~$R$.
Rappelons que le sous-groupe $U^1_q(K) \subset K_q(K)$ engendré par les symboles
$\{x_1,\dots,x_q\}$ où $x_1,\dots,x_q \in K^*$ sont tels que $x_1 \in 1+\mathfrak{m}$
est contenu dans~$U_q(K)$ et que l'on dispose d'un isomorphisme canonique $U_q(K)/U^1_q(K)=K_q(k)$ (\emph{loc.\ cit.}).
De plus,
si $K'/K$ est une extension finie,
la norme $N_{K'/K}:K_q(K')\to K_q(K)$ envoie $U_q(K')$ dans $U_q(K)$ et $U^1_q(K')$ dans $U^1_q(K)$
et l'application induite
$U_q(K')/U^1_q(K') \to U_q(K)/U^1_q(K)$
s'identifie à $eN_{k'/k}:K_q(k')\to K_q(k)$, où~$k'$ désigne le corps résiduel de~$K'$
et~$e$ l'indice de ramification de~$K'/K$
(cf. \cite[\textsection3.3]{katogeneralizationlcft}).
Comme~$k$ vérifie par hypothèse la propriété~$C_0^q$ en~$\ell$, la torsion $\ell$\nobreakdash-primaire du groupe $K_q(k)/eN_{k'/k}(K_q(k'))$ est annulée par~$e$.
D'autre part, comme~$\ell$ est inversible dans~$R$, le groupe $1+\mathfrak{m}$, et donc le groupe~$U^1_q(K)$, est divisible par~$\ell$.
Vu la suite exacte
\begin{align}
\label{se:u1uk}
\xymatrix@C=2em{
U^1_q(K) \ar[r] & U_q(K)/N_{K'/K}(U_q(K')) \ar[r] & K_q(k)/eN_{k'/k}(K_q(k')) \ar[r] & 0\rlap{\text{,}}
}
\end{align}
dont les deux termes de droite sont d'exposant fini, il s'ensuit que les exposants des groupes $U_q(K)/N_{K'/K}(U_q(K'))$ et
$K_q(k)/eN_{k'/k}(K_q(k'))$ ont même valuation $\ell$\nobreakdash-adique.
En conclusion, le groupe $\big(U_q(K)/N_{K'/K}(U_q(K'))\big)\otimes_\Z\Zl$ est annulé par~$e$ pour toute extension finie $K'/K$ d'indice de ramification~$e$.
Si maintenant $K'=K(x)$ pour un point fermé~$x$ de~$X$, alors $N_{K'/K}(U_q(K')) \subset U_q(K) \cap N_q(X/K)$.
Ainsi le groupe $\left(U_q(K)/(U_q(K)\cap N_q(X/K))\right) \otimes_\Z \Zl$
est-il annulé par le pgcd des
indices de ramification des extensions finies $K(x)/K$
lorsque~$x$ parcourt l'ensemble des points fermés de~$X$.
Il suffit donc, pour conclure, de montrer que ce pgcd
divise la multiplicité de~$Y$
à une puissance de~$p$ près, où~$p$ désigne la caractéristique de~$k$.
Mais cela résulte de \cite[\textsection9.1, Corollary~9 et Lemma~4]{BLR}.
\end{proof}

Étudions à présent le groupe apparaissant à droite dans~\eqref{se:decoupe}.

\begin{lem}
\label{lem:partialNq}
On a $\partial(N_q(X/K))=N_{q-1}(Y/k)$.
\end{lem}

Le lemme~\ref{lem:partialNq} vaut sans hypothèse sur~$k$.
Pour le démontrer, nous aurons besoin du lemme géométrique suivant.

\begin{lem}
\label{lem:geometrique}
Soit~$\sX$ un schéma régulier, plat et de type fini au-dessus
d'un anneau de valuation discrète hensélien excellent~$R$ de corps des fractions~$K$, de corps résiduel~$k$.
Notons~$X =\sX\otimes_RK$ et $Y =\sX\otimes_Rk$.  Si~$k$ est infini,
alors pour tout point fermé $y \in Y$, il existe un point fermé $x \in X$ tel que l'extension finie~$K(x)/K$ ait pour extension
résiduelle $k(y)/k$.
\end{lem}

\begin{proof}
Soit $y \in Y$ un point fermé.  Notons $f:\sX_1 \to \sX$ le schéma obtenu en faisant éclater le point~$y$.
Comme~$\sX$ est régulier, la fibre de~$f$ en~$y$ est un espace projectif.  Comme de plus~$k$ est infini, elle contient un point~$y_1$ ayant même corps résiduel
que~$y$ et n'appartenant
à aucune autre composante irréductible de~$\sX_1 \otimes_R k$.  Quitte à remplacer~$\sX$ par~$\sX_1$ et~$y$ par~$y_1$, on peut donc supposer le schéma~$Y_\red$ régulier en~$y$.

\newcommand{\OXy}{\sO_{\sX\mkern-4mu,\mkern1.5mu{}y}}
Soit alors $t \in \OXy$ tel que le diviseur $Y_\red \subset \sX$ ait pour équation~\mbox{$t=0$} au voisinage de~$y$.
Comme~$t$ n'est pas un diviseur de zéro dans~$\OXy$ et comme
$\OXy/(t)=\sO_{Y_\red,\mkern1.5mu{}y}$ est un anneau régulier,
l'élément~$t$ fait partie d'un système régulier de paramètres
$(t,f_1,\dots,f_n)$
de~$\OXy$ (cf.~\cite[Corollaire~17.1.8 et Proposition~17.1.7]{ega41}).
L'anneau $R'=\OXy/(f_1,\dots,f_n)$ est local, noethérien, régulier, de dimension~$1$: c'est un anneau de valuation discrète.
Notons~$\pi$ une uniformisante de~$R$. Comme il existe $m \geq 1$ tel que $t^m \in \pi \OXy$,
l'image de~$\pi$ dans~$R'$ est non nulle, de sorte que~$R'$ est plat sur~$R$.
D'autre part, comme~$R$ est hensélien,
comme~$R'$ est une $R$\nobreakdash-algèbre locale essentiellement de type fini et comme $R' \otimes_R k$ est de dimension finie sur~$k$,
la $R$\nobreakdash-algèbre~$R'$ est finie
(cf.~\cite[Théorème~18.5.11 $c'$]{ega44}).
En particulier, le morphisme canonique $\Spec(R')\to \sX$ est une immersion fermée
et envoie le point générique de~$\Spec(R')$ sur un point
fermé $x \in X$.  Le corps résiduel de~$R'$ étant~$k(y)$,
le lemme est prouvé.
\end{proof}

Les arguments du second paragraphe de la preuve du lemme~\ref{lem:geometrique} sont ceux de
\cite[\textsection9.1, Corollary~9]{BLR}.

\begin{proof}[Démonstration du lemme~\ref{lem:partialNq}]
L'inclusion $\partial(N_q(X/K)) \subset N_{q-1}(Y/k)$ résulte tout de suite de la compatibilité de la norme avec le résidu (cf.~\cite[Proposition~7.3.9]{gilleszamuely}),
compte tenu de ce que si $x \in X$ est un point fermé et si~$y$ désigne l'unique point fermé de~$Y$ appartenant à l'adhérence de~$x$ dans~$\sX$, l'extension $k(y)/k$ se plonge dans l'extension résiduelle de $K(x)/K$.

Si~$k$ est infini, le lemme~\ref{lem:geometrique}, la compatibilité de la norme avec le résidu et la surjectivité de l'application résidu entraînent ensemble l'inclusion réciproque.

Il reste à établir l'inclusion $N_{q-1}(Y/k) \subset \partial(N_q(X/K))$
lorsque~$k$ est fini.  Soit $\alpha \in N_{q-1}(Y/k)$.
Le groupe $K_{q-1}(k)/\partial(N_q(X/K))$ étant d'exposant fini
(cf.~\eqref{se:decoupe}), il existe $n \geq 1$ tel que $n\alpha \in \partial(N_q(X/K))$.
Soit $k=k_0 \subset k_1 \subset \cdots$ une suite d'extensions finies de~$k$ de degrés premiers à~$n$,
telle que le corps $k'=\bigcup_{i \geq 0}k_i$ soit infini.
Relevons-la en une suite $R=R_0 \subset R_1 \subset \cdots$ de $R$\nobreakdash-algèbres finies étales,
posons $R'=\bigcup_{i \geq 0}R_i$
et notons~$K_i$ le corps des fractions de~$R_i$ et~$K'$ celui de~$R'$.  Le corps~$k'$ étant infini, on a $N_{q-1}(Y \otimes_k k'/k') \subset \partial(N_q(X\otimes_KK'/K'))$.
Il existe donc $i \geq 0$ tel que l'image~$\alpha_i$ de~$\alpha$ dans $N_{q-1}(Y \otimes_k k_i/k_i)$ appartienne à $\partial(N_q(X\otimes_K K_i/K_i))$.
Comme $N_{k_i/k}(\partial(N_q(X \otimes_K K_i/K_i))) \subset \partial(N_q(X/K))$
et comme $N_{k_i/k}(\alpha_i)=[k_i:k]\mkern1mu\alpha$,
il s'ensuit que $[k_i:k]\mkern1mu\alpha \in \partial(N_q(X/K))$. D'où finalement $\alpha \in \partial(N_q(X/K))$ puisque~$[k_i:k]$ et~$n$ sont premiers entre eux.
\end{proof}

Nous sommes en position de conclure la preuve du théorème~\ref{th:transition}.
Notons~$m$ la multiplicité de~$Y$ et~$D$ le diviseur effectif
sur$~\sX$ tel que $Y=mD$.
Notons encore~$D$ le sous-schéma fermé de~$\sX$ défini par le faisceau d'idéaux $\sO_{\sX}(-D)$.
Comme~$D$ est un $k$\nobreakdash-schéma propre et comme le corps~$k$ vérifie la propriété~$C_1^{q-1}$ forte
en~$\ell$, le sous-groupe de torsion $\ell$\nobreakdash-primaire
de $K_{q-1}(k)/N_{q-1}(D/k)$ est annulé par $\chi(D,\sO_D)$.
Or $K_{q-1}(k)/N_{q-1}(D/k)=K_{q-1}(k)/N_{q-1}(Y/k)$
puisque $D_\red=Y_\red$.
Il s'ensuit, grâce au lemme~\ref{lem:partialNq},
que~$\chi(D,\sO_D)$ annule
le sous-groupe de torsion $\ell$\nobreakdash-primaire
du dernier terme de~\eqref{se:decoupe}.
D'autre part, d'après le lemme~\ref{lem:multipliciteU}, l'entier~$m$ annule le sous-groupe de torsion $\ell$\nobreakdash-primaire
du premier terme de~\eqref{se:decoupe}.  Par conséquent~$n_X$ divise le produit $m\chi(D,\sO_D)$.
Enfin, il est établi dans \cite[Proposition~2.4]{elw} que $\chi(X,\sO_X)=m\chi(D,\sO_D)$ (conséquence du théorème
de Snapper--Kleiman).  Ainsi~$n_X$ divise-t-il $\chi(X,\sO_X)$, comme il fallait démontrer.
\end{proof}

\begin{cor}
\label{cor:itere}
Soient~$n$ un entier naturel et~$p$ un nombre premier.
\begin{enumerate}
\item Si $n \geq 1$, le corps
$\C((x_1))\cdots((x_n))$ vérifie la propriété~$C_1^{n-1}$ forte.
\item Le corps $\F_p((x_1))\cdots((x_n))$ vérifie la propriété~$C_1^n$ forte
hors de~$p$.
\item Le corps $\Q_p((x_1))\cdots((x_n))$ vérifie la propriété~$C_1^{n+1}$ forte
hors de~$p$.
\end{enumerate}
\end{cor}

\begin{proof}
Les corps finis et le corps $\C((t))$ vérifient la propriété~$C_0^1$ d'après
\cite[Chapitre~X, \textsection7, Proposition~11]{serrecorpslocaux}.
D'autre part, ils vérifient la propriété~$C_1^0$ forte
d'après le corollaire~\ref{cor:kfini} et d'après \cite[Theorem~3.1]{elw}.
Le corollaire~\ref{cor:itere} en résulte
par une application répétée du théorème~\ref{th:transition},
compte tenu de la remarque~\ref{rem:transition}~(ii).
\end{proof}

\begin{exemple}
D'après le corollaire~\ref{cor:itere}, le corps $k=\C((x))((y))$, qui est de dimension cohomologique~$2$, vérifie la propriété~$C_1^1$ de Kato et Kuzumaki, confirmant
ainsi pour ce corps la conjecture de \cite[\textsection1]{katokuzumaki}.
Auparavant, seul le cas des hypersurfaces de degré premier était connu (\emph{loc.\ cit.}, \textsection3).
Remarquons que le
corollaire~\ref{cor:itere}
prédit l'égalité $k^*=N_1(X/k)$ pour les hypersurfaces $X \subset \P^n_k$ de degré~$d$ sous une hypothèse strictement plus faible que $d\leq n$;
il en va par exemple ainsi des surfaces sextiques dans~$\P^3_k$ (cf.~\cite[Lemma~4.3, Example~4.6]{elw}).
\end{exemple}

\begin{exemple}
D'après le corollaire~\ref{cor:itere}, le corps~$\Qp$ vérifie la propriété~$C_1^1$ forte hors de~$p$.  Il vérifie donc la propriété~$C_1^1$ de Kato et Kuzumaki
\og{}hors de~$p$\fg{}; cela n'était auparavant connu que dans le cas des hypersurfaces
de degré premier
(cf.~\cite[\textsection3]{katokuzumaki}).  Nous montrerons au~\textsection\ref{sec:corpspadiquesC11} que ce corps vérifie aussi la propriété~$C_1^1$ en~$p$.
Quant à la propriété~$C_1^1$ \emph{forte} hors de~$p$ pour~$\Qp$,
elle nous servira au~\textsection\ref{sec:cdn}.
\end{exemple}

\begin{remarques}
\label{rq:surlapreuve}
(i)
L'hypothèse que~$\ell$ est inversible dans~$R$ a été utilisée à deux reprises dans la preuve du théorème~\ref{th:transition}:
une première fois pour appliquer le théorème de Gabber--de~Jong, une autre fois dans la preuve du lemme~\ref{lem:multipliciteU}.
Néanmoins, si~$k$ est parfait de caractéristique $p>0$, le lemme~\ref{lem:multipliciteU} reste valable lorsque $\ell=p$.
En effet, si $K'/K$ est une extension finie,
si~$e$ désigne l'indice de ramification de~$K'/K$,
si~$K_0/K$ est la plus grande sous-extension
non ramifiée de~$K'/K$
et si~$k$ est parfait, alors $e=[K':K_0]$.
Il~s'ensuit que le groupe $U_q(K_0)/N_{K'/K_0}(U_q(K'))$ est annulé par~$e$.
D'autre part,
comme l'extension~$K_0/K$ est non ramifiée,
on a $U^1_q(K) \subset N_{K_0/K}(U_q(K_0))$ (d'après~\cite[Chapitre~IV, \textsection1, Proposition~3]{serrecorpslocaux}
et la formule de projection);
comme de plus~$k$ vérifie la propriété~$C_0^q$ en~$\ell$,
la suite exacte~\eqref{se:u1uk} associée à l'extension~$K_0/K$ montre
que la $\ell$\nobreakdash-torsion du groupe $U_q(K)/N_{K_0/K}(U_q(K_0))$ est nulle.
Ces remarques impliquent que~$e$ annule le sous-groupe de torsion $\ell$\nobreakdash-primaire de $U_q(K)/N_{K'/K}(U_q(K'))$. La suite de la démonstration du lemme~\ref{lem:multipliciteU}
reste inchangée.

(ii)
Si l'on convient que $K_{-1}(k)=0$, on peut autoriser $q=0$ dans l'énoncé du théorème~\ref{th:transition}.  La preuve se simplifie;
on retrouve ainsi \cite[Theorem~3.1]{elw}.

(iii)
Compte tenu de l'existence de résolutions des singularités pour les schémas excellents de dimension~$\leq2$ (cf.~\cite{lipman})
et des remarques~\ref{rq:devissage}~(ii) et~\ref{rq:surlapreuve}~(i),
le corps~$\Qp$ vérifie aussi la propriété~$C_1^1$ forte en~$p$ restreinte aux schémas de dimension~$\leq 1$.
Ainsi par exemple, si~$X$ est une courbe propre de genre arithmétique~$2$ sur~$\Qp$, le groupe $N_1(X/\Qp)$ est égal à~$\Qp^\star$ tout entier.
\end{remarques}

\section{Les corps \texorpdfstring{$p$}{p}-adiques sont \texorpdfstring{$C_1^1$}{C\_1\textasciicircum1}}
\label{sec:corpspadiquesC11}

Bien que,
comme nous l'avons vu au \textsection\ref{sec:transition},
le corps~$\Qp$ vérifie la propriété~$C_1^1$ forte hors de~$p$,
nous ne savons pas s'il vérifie en toute généralité la propriété~$C_1^1$ forte, faute de disposer de la résolution des singularités pour les schémas de type fini sur~$\Z_p$.
Le~but de ce paragraphe est de démontrer que~$\Qp$ satisfait néanmoins la propriété~$C_1^1$, confirmant ainsi une conjecture de Kato
et Kuzumaki.

\begin{definition}
Soit~$K$ le corps des fractions d'un anneau de valuation discrète hensélien excellent.
Si~$X$ est un schéma de type fini sur~$K$,
l'\emph{indice résiduel de~$X$ sur~$K$} est le pgcd des degrés résiduels des extensions finies $K(x)/K$
lorsque~$x$ parcourt l'ensemble des points fermés de~$X$.
\end{definition}

L'indice résiduel de~$X$ sur~$K$ divise l'indice de~$X$ sur~$K$.
Avec les notations du~\textsection\ref{sec:transition}, on voit aisément que si~$X$ est non vide, l'indice résiduel de~$X$ sur~$K$ est l'ordre
du groupe cyclique $K_0(k)/\partial(N_1(X/K))$.

\begin{thm}
\label{th:indresiduel}
Soit~$K$ un corps $p$\nobreakdash-adique, c'est-à-dire une extension finie de~$\Qp$.
Pour tout $K$\nobreakdash-schéma propre~$X$
et tout faisceau cohérent~$E$ sur~$X$,
l'indice résiduel de~$X$ sur~$K$ divise $\chi(X,E)$.
\end{thm}

Le point clef de la démonstration
du théorème~\ref{th:indresiduel} consiste à
vérifier que le lemme~\ref{lem:partialNq} (resp.~le lemme~\ref{lem:geometrique}) reste valable en présence de singularités
quotient, du moins lorsque~$k$ est un corps fini
(resp.~une extension algébrique infinie d'un corps fini, ou plus généralement un corps pseudo-algébriquement clos).
C'est le contenu du lemme~\ref{lem:quotient} ci-dessous.

\begin{proof}[Démonstration du théorème~\ref{th:indresiduel}]
Notons~$R$ l'anneau des entiers de~$K$ et~$k$ le corps résiduel de~$R$.
Convenons qu'un $K$\nobreakdash-schéma propre~$X$
possède la propriété~$\PP$
s'il existe un groupe fini~$G$ et un $R$\nobreakdash-schéma~$\sX'$ irréductible, régulier, projectif et plat, muni d'une action
de~$G$ (étant entendu que le
morphisme $\sX'\to\Spec(R)$ est équivariant pour l'action triviale de~$G$ sur~$R$),
tels que~$X$ soit isomorphe à la fibre générique de $\sX=\sX'/G$ au-dessus de~$R$.
Cette définition est motivée par le lemme suivant.

\begin{lem}
\label{lem:quotient}
Soit~$G$ un groupe fini.  Soit~$\sX'$ un $R$\nobreakdash-schéma régulier, projectif et plat, muni d'une action de~$G$.
Notons $\sX=\sX'/G$ le quotient et posons $X=\sX\otimes_RK$ et $Y=\sX\otimes_Rk$.
Alors l'indice résiduel de~$X$ sur~$K$ est égal à l'indice de~$Y$ sur~$k$.
\end{lem}

\begin{proof}
Si $x \in X$ est un point fermé et si~$y$ désigne l'unique point fermé de~$Y$ appartenant à l'adhérence de~$x$ dans~$\sX$, le degré de~$y$
divise le degré résiduel de~$x$.  Par conséquent, l'indice de~$Y$ sur~$k$ divise l'indice résiduel de~$X$ sur~$K$.
Pour établir la divisibilité réciproque, fixons un point fermé $y \in Y$ et montrons que l'indice résiduel de~$X$ sur~$K$ divise le degré de~$y$.

Notons $q:\sX'\to \sX$ le morphisme canonique
et $f':\sX'_1\to\sX'$ le schéma obtenu en faisant éclater $q^{-1}(y)_\red$ dans~$\sX'$.
Le groupe~$G$ agit naturellement sur~$\sX'_1$ et le morphisme~$f'$ est $G$\nobreakdash-équivariant.
Posons $\sX_1=\sX'_1/G$ et notons $q_1:\sX'_1\to \sX_1$ le morphisme canonique et $f:\sX_1\to\sX$ le morphisme induit par~$f'$.  Notons de plus
$E'=f'^{-1}(q^{-1}(y)_\red) \subset \sX'_1$
et $E=q_1(E') \subset \sX_1$.

Comme~$\sX'$ et $q^{-1}(y)_\red$ sont réguliers, le diviseur exceptionnel~$E'$ est un espace projectif au-dessus de~$q^{-1}(y)_\red$.
En particulier, si~$L$ désigne un corps algébriquement clos contenant~$k(y)$, les composantes irréductibles de $E' \otimes_{k(y)} L$ sont en
bijection avec les $L$\nobreakdash-points de $q^{-1}(y)$.
Or~$G$ agit transitivement sur $q^{-1}(y)(L)$
(cf.~\cite[Corollary~A7.2.2]{katzmazur}). Par conséquent $(E' \otimes_{k(y)} L)/G$ est un schéma irréductible.
Comme $(E'\otimes_{k(y)}L)/G=(E'/G)\otimes_{k(y)}L$, on conclut que $E'/G$ est géométriquement irréductible sur~$k(y)$,
puis que~$E$ est géométriquement irréductible sur~$k(y)$, étant donné que~$E$ est l'image de~$E'/G$ dans~$\sX_1$.

Le fermé $E \subset \sX_1$
est de codimension~$1$
puisque~$q_1$ est fini et que~$E'$ est de codimension~$1$ dans~$\sX_1'$
(cf.~\cite[Proposition~5.4.2]{ega42}).
Par ailleurs, comme~$\sX_1'$ est normal, le schéma~$\sX_1$ est normal; en particulier,
il est régulier en dehors d'un fermé de codimension~$\geq 2$
(\emph{op.\ cit.}, Corollaire~6.12.6).
Il existe donc un ouvert dense $E^0 \subset E$ tel que~$\sX_1$ soit régulier en tout point de~$E^0$.
Quitte à rétrécir~$E^0$, on peut supposer~$E^0$ régulier et disjoint de toute composante irréductible de $Y_1=\sX_1\otimes_Rk$ autre que~$E$.

Le corps~$k(y)$ étant fini et la variété~$E^0$ géométriquement irréductible sur~$k(y)$,
il existe, d'après les bornes de Lang--Weil--Nisnevi\v{c} (cf.~\cite{langweil}, \cite{nisnevic}),
un point fermé $y^0 \in E^0$ tel que $[k(y^0):k(y)]$ soit premier à l'indice résiduel de~$X$ sur~$K$.
Par construction de~$E^0$, les schémas~$\sX_1$ et $(Y_1)_\red$ sont réguliers en~$y^0$.
La seconde partie de la démonstration du lemme~\ref{lem:geometrique} s'applique donc mot à mot et fournit un point fermé $x \in X = \sX_1 \otimes_RK$
tel que l'extension finie $K(x)/K$ ait pour extension résiduelle $k(y^0)/k$.
Il s'ensuit que l'indice résiduel de~$X$ sur~$K$ divise $[k(y):k]$, comme il fallait démontrer.
\end{proof}

Pour conclure la démonstration du théorème~\ref{th:indresiduel}, vérifions les propriétés~(1) à~(3) de la proposition~\ref{prop:devissage},
en notant~$n_X$ l'indice résiduel de~$X$ sur~$K$
pour tout $K$\nobreakdash-schéma propre~$X$.
La propriété~(1) est évidente.
La propriété~(2) résulte de la combinaison du lemme~\ref{lem:quotient} et du corollaire~\ref{cor:kfini}, en remarquant que dans la situation
du lemme~\ref{lem:quotient}, on a $\chi(X,\sO_X)=\chi(Y,\sO_Y)$
puisque~$\sX$ est plat sur~$R$.
Montrons la propriété~(3).
Si~$X$ est un $K$\nobreakdash-schéma propre et intègre, notons~$K_1$ la fermeture algébrique de~$K$ dans~$K(X)$ et~$X_1$ la normalisation de~$X$.
Le schéma~$X_1$ est naturellement un $K_1$\nobreakdash-schéma
propre, intègre et géométriquement irréductible.
Comme~$K$ est de caractéristique nulle,
la version équivariante du théorème
de de~Jong~\cite[Theorem~5.9]{dejonggrenoble}
appliquée à un modèle propre de~$X_1$ au-dessus de l'anneau des entiers de~$K_1$
fournit une modification de~$X$, c'est-à-dire un schéma intègre muni d'un morphisme
propre et birationnel vers~$X$, qui possède la propriété~$\PP$.
\end{proof}

\begin{cor}
\label{cor:padique}
Soit~$K$ un corps $p$\nobreakdash-adique.
Soient~$X$ un $K$\nobreakdash-schéma propre et~$E$ un faisceau cohérent sur~$X$.
Notons~$K^\nr$ l'extension non ramifiée maximale de~$K$.
Si~$X(K^\nr)\neq\emptyset$, le groupe $K^*/N_1(X/K)$ est annulé par $\chi(X,E)$.
\end{cor}

\begin{proof}
La norme étant surjective sur les unités dans toute extension finie non ramifiée de corps $p$\nobreakdash-adiques,
l'hypothèse $X(K^\nr)\neq\emptyset$ assure que le terme de gauche de la suite exacte~\eqref{se:decoupe} pour $q=1$ est nul.
Le terme de droite étant par ailleurs annulé par l'indice résiduel de~$X$ sur~$K$,
le théorème~\ref{th:indresiduel} permet de conclure.
\end{proof}

\begin{cor}
\label{cor:padiquesC11}
Les corps $p$\nobreakdash-adiques sont~$C_1^1$.
\end{cor}

\begin{proof}
Les hypersurfaces $X\subset \P^n$ de degré $d\leq n$ vérifient $\chi(X,\sO_X)=1$ et possèdent un point
sur l'extension non ramifiée maximale de tout corps $p$\nobreakdash-adique
d'après un théorème de Lang (cf.~\cite{langqac}).  Le corollaire~\ref{cor:padique} s'applique.
\end{proof}

Kato et Kuzumaki
avaient établi le corollaire~\ref{cor:padiquesC11} pour les hypersurfaces de degré premier
(cf.~\cite[\textsection3, Corollary~1]{katokuzumaki})
et conjecturé sa validité en général
(\emph{op.\ cit.}, \textsection5, Problem~3).

\begin{cor}
\label{cor:espacehomogene}
Si~$K$ est un corps $p$\nobreakdash-adique et~$X$ un espace homogène d'un groupe algébrique linéaire connexe sur~$K$,
alors $K^*=N_1(X/K)$.
\end{cor}

Dans cet énoncé, l'espace~$X$ n'est supposé ni projectif ni principal homogène.

\begin{proof}
Soit $X \subset X'$ une compactification lisse de~$X$.
Comme la variété~$X'$ est géométriquement unirationnelle (cf.~\cite{chevalley}),
elle vérifie $\chi(X',\sO_{X'})=1$.
Elle admet par ailleurs un point
dans l'extension non ramifiée maximale de~$K$
d'après un théorème de Springer et Steinberg,
cette extension étant un corps de dimension cohomologique~$1$
(cf.~\cite[Chapitre~III, \textsection2.3 et~\textsection2.4]{serrecg}).
Le corollaire~\ref{cor:padique} entraîne donc que $K^*=N_1(X'/K)$.
D'autre part, d'après le théorème des fonctions implicites,
la lissité de~$X'$ implique que~$X(L)$ est dense dans~$X'(L)$ pour toute extension finie~$L/K$,
de sorte que $N_1(X'/K)=N_1(X/K)$.
\end{proof}

\begin{remarque}
\label{rq:invbir}
De façon générale, quel que soit le corps~$k$, les groupes $N_q(X/k)$
sont des invariants birationnels des $k$\nobreakdash-schémas lisses de type fini.
En effet, si~$X$ est un ouvert dense d'un tel schéma~$X'$,
si $x \in X'$ est un point fermé, si $C \subset X' \otimes_kk(x)$ est une courbe
régulière passant par~$x$ et rencontrant~$X$ et si $Z \subset C \cap (X \otimes_kk(x))$ est le support d'un diviseur sur~$C$ de degré~$1$ sur~$k(x)$,
alors $K_q(k(x))=N_q(Z/k(x))$.
\end{remarque}

Lorsque~$k$ est un corps de nombres,
un théorème de Kato et Saito~\cite[\textsection7, Theorem~4]{katosaito}
affirme que l'application naturelle
\begin{align}
\label{eq:localglobal}
k^*/N_1(X/k) \to \smash{\prod_{v \in \Omega}}\vphantom{\prod_v} k_v^*/N_1(X\otimes_k k_v/k_v)
\end{align}
est un isomorphisme
si~$X$ est une variété
projective, lisse et géométriquement irréductible sur~$k$ et si~$\Omega$ désigne l'ensemble des places de~$k$ et~$k_v$ le complété de~$k$ en $v\in\Omega$.
D'après la remarque~\ref{rq:invbir}, l'hypothèse de projectivité est superflue.
Compte tenu de ce théorème, le corollaire~\ref{cor:espacehomogene} entraîne donc tout de suite le

\begin{cor}
\label{cor:espacehomogenecdn}
Si~$k$ est un corps de nombres, si~$X$ est un espace homogène d'un groupe algébrique linéaire connexe sur~$k$
et si $X(k_v)\neq\emptyset$ pour toute place réelle~$v$ de~$k$, alors $k^*=N_1(X/k)$.
\end{cor}

Plusieurs cas particuliers des corollaires~\ref{cor:espacehomogene} et~\ref{cor:espacehomogenecdn} sont bien connus ou se trouvent dans la littérature.
Lorsque~$X$ est une variété de Severi--Brauer, on retrouve notamment le théorème de Hasse--Schilling--Maass (\cite{hasseschilling}, \cite{maass})
selon lequel
la norme réduite d'une algèbre centrale simple sur
un corps de nombres totalement imaginaire
est surjective
(cf.~\cite[\textsection10, Lemma~7]{katosaito}).
Lorsque~$X$ est une quadrique, on retrouve la surjectivité de la norme spinorielle pour les formes quadratiques non dégénérées de rang au moins~$3$ définies
sur un corps $p$\nobreakdash-adique ou sur un corps de nombres totalement imaginaire, grâce au principe de norme de Knebusch (cf.~\cite{knebusch}, \cite[Satz~A]{kneser}).
Enfin, le cas où~$X$ est la variété des sous-groupes de Borel d'un groupe réductif fixé est traité dans \cite[Lemma~III.2.8]{gillerequivalence}.

\section{Les corps de nombres totalement imaginaires sont~\texorpdfstring{$C_1^1$}{C\_1\textasciicircum1}}
\label{sec:cdn}

Le but de ce paragraphe est d'établir la conjecture de~\cite{katokuzumaki} sur la propriété~$C_1^1$ pour les corps de nombres totalement imaginaires.

\begin{thm}
\label{th:cdntotimaginaires}
Les corps de nombres totalement imaginaires sont~$C_1^1$.
\end{thm}

Le théorème~\ref{th:cdntotimaginaires} était connu dans le cas des hypersurfaces \emph{lisses} de degré premier (cf.~\cite[\textsection4, Theorem~2]{katokuzumaki}).
Kato et Kuzumaki l'avaient déduit de leurs résultats sur les corps $p$\nobreakdash-adiques grâce au
théorème local-global
rappelé à la fin du~\textsection\ref{sec:corpspadiquesC11}.
Ce~théorème local-global s'applique à des variétés lisses et géométriquement irréductibles.
L'hypothèse d'irréductibilité géométrique est cruciale:
il est en effet bien connu, depuis Hasse~\cite{hassecyclic}, que
même si~$X$ est le spectre d'une extension biquadratique de~$k$,
l'application~\eqref{eq:localglobal} n'est pas injective en général.
Si l'on cherche à établir la propriété~$C_1^1$ sans hypothèse de lissité
pour les corps de nombres totalement imaginaires à l'aide de ce principe
local-global et de la propriété~$C_1^1$ des corps $p$\nobreakdash-adiques,
un dévissage du type envisagé au~\textsection\ref{sec:devissage} paraît donc nécessaire.  Un tel dévissage
fait cependant apparaître des variétés \emph{a priori} quelconques, qui n'ont pas de raison d'être des hypersurfaces;
or nous savons seulement que les complétés~$k_v$ vérifient la propriété~$C_1^1$ et non la propriété~$C_1^1$ forte.
Nous contournerons cette difficulté,
dans la démonstration du théorème~\ref{th:cdntotimaginaires},
en combinant le corollaire~\ref{cor:padiquesC11}
avec la remarque que~$k_v$ vérifie néanmoins
la propriété~$C_1^1$ forte restreinte aux schémas de dimension $<p-1$,
où~$p$ désigne la caractéristique résiduelle de~$v$ (conséquence du théorème de Hirzebruch--Riemann--Roch et du corollaire~\ref{cor:itere}; cf.~la démonstration de la proposition~\ref{prop:cdnQ} ci-dessous).

\begin{proof}[Démonstration du théorème~\ref{th:cdntotimaginaires}]
Soit~$k$ un corps de nombres.  Notons~$\Omega$ l'ensemble de ses places.
Si~$X$ est un $k$\nobreakdash-schéma propre et $S \subset \Omega$ un ensemble fini, posons
\begin{align}
Q_{X/k,S} = \Ker\Big(k^*/N_1(X/k) \to \prod_{v \in S} k_v^*/N_1(X_v/k_v)\Big)\rlap{\text{,}}
\end{align}
où $X_v=X\otimes_k k_v$ et où~$k_v$ désigne le complété de~$k$ en~$v$.

\begin{prop}
\label{prop:cdnQ}
Soit~$n$ un entier.
Supposons que~$S$ contienne les places réelles de~$k$ et les places finies de caractéristique résiduelle~$\leq n$.
Alors pour tout $k$\nobreakdash-schéma propre~$X$ tel que $\dim(X)<n$
et pour tout faisceau cohérent~$E$ sur~$X$,
le groupe $Q_{X/k,S}$ est annulé par $\chi(X,E)$.
\end{prop}

\begin{proof}
Pour tout $k$\nobreakdash-schéma propre~$X$, notons~$n_X$ l'exposant de $Q_{X/k,S}$ si~$X$
est non vide ou~$0$ sinon, et convenons que~$X$ possède la propriété~$\PP$
si~$X$ est irréductible et lisse.
Compte tenu de la proposition~\ref{prop:devissage} et de la remarque~\ref{rq:devissage}~(ii), il suffit, pour démontrer la proposition~\ref{prop:cdnQ},
de vérifier les propriétés~(1) à~(3) de la proposition~\ref{prop:devissage} pour les schémas de dimension~$<n$.
La propriété~(3) est satisfaite grâce à Hironaka.
La~propriété~(1) est conséquence du lemme suivant.

\begin{lem}
\label{lem:surjqyqx}
Pour tout morphisme
de $k$\nobreakdash-schémas propres non vides
$Y \to X$, l'application naturelle $Q_{Y/k,S} \to Q_{X/k,S}$
est surjective.
\end{lem}

\begin{proof}
D'après Kato et Saito~\cite[Lemma~12]{katosaito},
l'image de l'application diagonale $N_1(X/k) \to \prod_{v \in S} N_1(X_v/k_v)$ est dense.
Tout élément de $Q_{X/k,S}$ est donc représenté par un élément de~$k^*$
arbitrairement proche de~$1$ aux places de~$S$.
D'autre part,
pour toute place $v\in S$,
le sous-groupe $N_1(Y\otimes_kk_v/k_v)$ de~$k_v^*$ est ouvert
puisqu'il contient $k_v^{*m}$ pour un $m \geq 1$.
Le lemme résulte de la combinaison de ces deux affirmations.
\end{proof}

Un lemme similaire nous sera nécessaire pour établir la propriété~(2).

\begin{lem}
\label{lem:kkprime}
Soit~$k'/k$ une extension finie.  Soit~$X$ un $k'$\nobreakdash-schéma propre.
Notons~$S'$ l'ensemble des places de~$k'$ divisant une place de~$S$.
Le conoyau de l'application norme
$Q_{X/k',S'} \to Q_{X/k,S}$
est annulé par $[k':k]$.
\end{lem}

\begin{proof}
Supposons~$X$ non vide (lorsque~$X$ est vide, le lemme est trivial puisque $N_1(X/k)=1$ et $N_1(X_v/k_v)=1$ pour tout~$v$).
Comme on l'a vu dans la démonstration du lemme~\ref{lem:surjqyqx},
tout élément de $Q_{X/k,S}$ est représenté par un élément de~$k^*$ dont l'image dans $\prod_{w \in S'}k'^*_w$
est arbitrairement proche de~$1$ et en particulier appartient à $N_1(X\otimes_{k'} k'_w/k'_w)$ pour tout $w\in S'$.
Par conséquent, tout élément de $Q_{X/k,S}$ est représenté par un élément de~$k^*$ dont l'image dans $k'^*/N_1(X/k')$
appartient à~$Q_{X/k',S'}$.
Le lemme~\ref{lem:kkprime} s'ensuit en appliquant la norme de~$k'$ à~$k$.
\end{proof}

Fixons maintenant un $k$\nobreakdash-schéma~$X$ irréductible, propre, lisse,
de dimension~$<n$ et montrons que $Q_{X/k,S}$ est annulé par $\chi(X,\sO_X)$.
Compte tenu de la formule~\eqref{eq:multiplicativitedim} et de la normalité de~$X$,
le lemme~\ref{lem:kkprime}
permet de supposer~$X$ géométriquement irréductible sur~$k$,
quitte à remplacer~$k$ par sa fermeture algébrique~$k'$ dans~$k(X)$.
Le $k$\nobreakdash-schéma~$X$ étant à présent lisse et géométriquement irréductible,
le théorème local-global de Kato et Saito~\cite[\textsection7, Theorem~4]{katosaito}
fournit\footnote{Une hypothèse de projectivité apparaît dans \cite[\textsection7, Theorem~4]{katosaito}.  Ni cette hypothèse, ni celle, plus faible, de propreté,
n'interviennent cependant dans la démonstration.  Alternativement, que cette hypothèse soit superflue résulte aussi de la remarque~\ref{rq:invbir}.} un isomorphisme
\begin{align}
\label{eq:katosaito}
Q_{X/k,S} \longisoto \prod_{v \in \Omega \setminus S} k_v^*/N_1(X_v/k_v)\rlap{\text{.}}
\end{align}

Soit $v \in \Omega \setminus S$ une place finie.  Notons~$p$ sa caractéristique résiduelle.
Comme $\dim(X)<n$, le théorème de Hirzebruch--Riemann--Roch entraîne que l'indice de~$X$ sur~$k$
divise $\chi(X,\sO_X)$ dans $\Z[1/n!]$ (cf.~\cite[Proposition~1.2]{elw}).  Il s'ensuit que
le groupe $\left(k_v^*/N_1(X_v/k_v)\right) \otimes_\Z \Z[1/n!]$ est annulé par $\chi(X,\sO_X)$.
D'autre part, puisque le corps~$k_v$ vérifie la propriété~$C_1^1$ forte hors de~$p$
(cf.~corollaire~\ref{cor:itere}),
le groupe $\left(k_v^*/N_1(X_v/k_v)\right) \otimes_\Z \Z[1/p]$ est annulé par $\chi(X,\sO_X)$.
L'hypothèse faite sur~$S$ assure que $p>n$;
le groupe $k_v^*/N_1(X_v/k_v)$ est donc lui-même annulé par $\chi(X,\sO_X)$.
La~place~$v$ étant quelconque parmi les places finies hors de~$S$
et le groupe $k_v^*/N_1(X_v/k_v)$ étant nul pour~$v$ complexe,
on conclut, grâce à~\eqref{eq:katosaito}, que $Q_{X/k,S}$ est annulé par $\chi(X,\sO_X)$.
La proposition~\ref{prop:cdnQ} est ainsi prouvée.
\end{proof}

Établissons maintenant le théorème~\ref{th:cdntotimaginaires}.  Supposons~$k$ totalement imaginaire et fixons une hypersurface $X \subset \P^n_k$ de degré $d \leq n$.
Soit~$S$ l'ensemble des places finies de~$k$ de caractéristique résiduelle~$\leq n$.  Considérons la suite exacte
\begin{align}
\label{se:localglobal}
\xymatrix{
0 \ar[r] & Q_{X/k,S} \ar[r] & k^*/N_1(X/k) \ar[r] & \displaystyle\prod_{v \in S} k_v^*/N_1(X_v/k_v)\rlap{\text{.}}
}
\end{align}
Comme $\chi(X,\sO_X)=1$,
la proposition~\ref{prop:cdnQ} montre que le terme de gauche est nul.  Le terme de droite est nul d'après
le corollaire~\ref{cor:padiquesC11}.  Il s'ensuit que $k^*=N_1(X/k)$, comme il fallait démontrer.
\end{proof}

\begin{remarque}
Lorsque~$k$ est un corps de nombres formellement réel, par exemple~$\Q$,
il résulte de la démonstration du théorème~\ref{th:cdntotimaginaires}
que l'égalité $k^*=N_1(X/k)$ vaut pour toute hypersurface $X \subset \P^n_k$ de degré~$d \leq n$
vérifiant $X(k_v)\neq\emptyset$ pour toute place réelle~$v$ de~$k$.
\end{remarque}

\section{Questions, exemples et remarques}
\label{sec:qrem}

\subsection{Cohomologie galoisienne et faisceaux cohérents}
\label{sec:cohgaletfaisecauxcoh}

La définition de la propriété~$C_1^q$ forte et les résultats des~\textsection\ref{sec:ax} à~\ref{sec:cdn}
suggèrent la généralisation suivante du cas $i=1$ de \cite[\textsection5, Problem~1]{katokuzumaki},
qui concernait les hypersurfaces de degré~$d \leq n$ dans~$\P^n_k$.

\begin{question}
\label{q:noyau}
Soit~$k$ un corps non ordonnable (cf.~\cite[\textsection5.1]{jacobson}).  Soit~$p$ un nombre premier inversible dans~$k$.
Soient~$X$ un $k$\nobreakdash-schéma propre et~$E$ un faisceau cohérent sur~$X$.
Le noyau de l'application naturelle
\begin{align}
\label{eq:apprest}
H^1(k,\Z/p\Z) \to \prod H^1(k(x),\Z/p\Z) \rlap{\text{,}}
\end{align}
où~$x$ parcourt l'ensemble des points fermés de~$X$, est-il annulé par $\chi(X,E)$?
\end{question}

La question~\ref{q:noyau} admet une réponse affirmative
si~$X$ est une variété de Severi--Brauer (cf.~\cite[\textsection5, Lemma~6]{katokuzumaki}).
D'autre part, sans hypothèse sur~$X$, elle
admet une réponse affirmative
lorsque $k=\C((t))$ (resp.\ lorsque~$k$ est un corps fini),
comme il résulte de~\cite[Theorem~1]{elw} (resp.\ du corollaire~\ref{cor:kfini}),
et plus généralement
lorsque~$k$ est l'un des corps apparaissant dans le corollaire~\ref{cor:itere} et que~$p$ n'est pas la caractéristique résiduelle.
Cela découle du corollaire~\ref{cor:itere} et de la preuve
de \cite[\textsection5, Proposition~3]{katokuzumaki},
qui repose sur les théorèmes de dualité locale en cohomologie galoisienne.
À l'aide d'arguments plus élémentaires, nous
montrons, dans la proposition~\ref{prop:hilbertien} ci-dessous, que la question~\ref{q:noyau} admet également une réponse affirmative dès que~$k$ remplit l'une des conditions suivantes
(et ce, même si~$k$ est ordonnable):
\begin{enumerate}
\item[(i)] $k$ est un corps de nombres;
\item[(ii)] il existe un sous-corps $k_0 \subset k$ tel que l'extension $k/k_0$ soit de type fini mais ne soit pas algébrique (autrement dit~$k$ est un corps de fonctions);
\item[(iii)] $k$ est le corps des fractions d'un anneau factoriel non principal (par exemple $k=k_0((x_1,\dots,x_n))$ pour un entier $n \geq 2$ et un corps~$k_0$).
\end{enumerate}
Ces trois types de corps partagent en effet la propriété d'être hilbertiens (cf.~\cite[Theorem~13.4.2, Theorem~15.4.6]{friedjarden}).

\begin{prop}
\label{prop:hilbertien}
La question~\ref{q:noyau} admet une réponse affirmative lorsque~$k$ est un corps hilbertien.
\end{prop}

\begin{proof}
Nous aurons besoin du lemme suivant.

\begin{lem}
\label{lem:restinj}
Soit~$X$ un schéma normal, de type fini sur un corps hilbertien~$k$.
Soit~$p$ un nombre premier.
L'application naturelle
\begin{align}
\label{eq:appnat}
H^1(X,\Z/p\Z) \to \prod H^1(k(x),\Z/p\Z)\rlap{\text{,}}
\end{align}
où~$x$ parcourt l'ensemble des points fermés de~$X$, est injective.
\end{lem}

\begin{proof}
On peut supposer~$X$ connexe et non vide.
Sous cette hypothèse, tout élément non nul~$\alpha$ de $H^1(X,\Z/p\Z)$ est représenté par un revêtement étale connexe $\pi:X' \to X$ de degré~$p$.
Les schémas~$X$ et~$X'$ étant normaux, connexes et non vides, ils sont irréductibles.
Soit~$U \subset X$ un ouvert dense
muni d'un morphisme quasi-fini et dominant
$f:U \to \mathbf{A}^n_k$.
Comme~$k$ est hilbertien, il existe $z \in \mathbf{A}^n_k(k)$ tel que
les schémas $f^{-1}(z)$ et $\pi^{-1}(f^{-1}(z))$ soient irréductibles.
Posant $x=f^{-1}(z)$, l'image de~$\alpha$ dans $H^1(k(x),\Z/p\Z)$ est alors non nulle.
\end{proof}

Pour tout $k$\nobreakdash-schéma propre~$X$,
notons~$n_X$ l'exposant du noyau de~\eqref{eq:apprest}
et appelons~$\PP$ la propriété, pour~$X$, d'être irréductible et normal.
Dans ce contexte, les propriétés~(1) et~(3) de la proposition~\ref{prop:devissage} sont évidentes.
La propriété~(2) découle du lemme~\ref{lem:restinj} et de la remarque que si~$X$ est un $k$\nobreakdash-schéma propre, irréductible et normal
et si~$k'$ désigne la fermeture algébrique de~$k$
dans~$k(X)$,
le noyau de l'application naturelle $H^1(k,\Z/p\Z)\to H^1(X,\Z/p\Z)$ est annulé par $[k':k]$,
donc par~$\chi(X,\sO_X)$
au vu de la formule~\eqref{eq:multiplicativitedim}.
D'où le résultat, grâce à la proposition~\ref{prop:devissage}.
\end{proof}

En particulier, le problème d'origine \cite[\textsection5, Problem~1]{katokuzumaki} admet une
solution positive, pour $i=1$, pour tous les corps mentionnés
au~\textsection\ref{sec:cohgaletfaisecauxcoh}.
L'exemple de la courbe (irréductible, avec un point double ordinaire) obtenue
à partir de~$\P^1_k$ en identifiant les points~$0$ et~$\infty$
montre que l'énoncé du lemme~\ref{lem:restinj} tombe en défaut si~$X$ n'est plus supposé normal
(dans cet exemple, quels que soient~$p$ et~$k$, l'application~\eqref{eq:appnat}
se factorise par l'application
 image
réciproque $H^1(X,\Z/p\Z)\to H^1(\P^1_k,\Z/p\Z)$
et celle-ci
 a pour noyau~$\Z/p\Z$).
Le passage par la question~\ref{q:noyau} et par
le principe de dévissage du~\textsection\ref{sec:devissage} semble donc
inévitable pour répondre à \cite[\textsection5, Problem~1]{katokuzumaki}
pour $i=1$ et~$k$ hilbertien en toute généralité.


\subsection{Un exemple de corps \texorpdfstring{$C_1^0$}{C\_1\textasciicircum{}0} qui n'est pas \texorpdfstring{$C_1$}{C\_1}.}
\label{sec:exemple}

Tout corps~$C_1$ vérifie de façon évidente la propriété~$C_1^0$ de Kato et Kuzumaki.
En nous appuyant sur le principe de dévissage du~\textsection\ref{sec:devissage},
nous donnons dans ce paragraphe un contre-exemple à l'implication réciproque.
Plus précisément,
Ax~\cite{ax} a construit
un corps de dimension cohomologique~$1$ et de caractéristique~$0$ qui n'est pas~$C_1$.
Nous prouvons ci-dessous que ce corps vérifie la propriété~$C_1^0$ forte, \emph{a fortiori}
la propriété~$C_1^0$.

Commençons par une proposition générale concernant la propriété~$C_1^0$ forte pour les corps
de séries de Puiseux.

\begin{prop}
\label{prop:C10kK}
Si~$k$ est un corps de caractéristique~$0$ vérifiant la propriété~$C_1^0$ forte,
le corps $K=\bigcup_{n \geq 1} k((t^{1/n}))$ vérifie aussi la propriété~$C_1^0$ forte.
\end{prop}

\begin{proof}
Si~$X$ est un $K$\nobreakdash-schéma propre, notons~$n_X$ l'indice de~$X$ sur~$K$.
D'après la proposition~\ref{prop:devissage},
compte tenu du lemme de Chow (pour la projectivité)
et du théorème de Hironaka (pour la lissité),
il suffit de montrer que pour tout $K$\nobreakdash-schéma~$X$ irréductible, projectif
et lisse, l'indice de~$X$ sur~$K$ divise $\chi(X,\sO_X)$.
Fixons un tel~$X$ et posons $R_n=k[[t^{1/n}]]$ pour $n\geq 1$.

\begin{lem}
\label{lem:redsemistable}
Il existe un entier $n \geq 1$ et un $R_n$\nobreakdash-schéma~$\sX$ régulier, projectif et plat tel que $\sX \otimes_{R_n}K=X$ et tel que
le schéma $Y=\sX \otimes_{R_n}k$ soit réduit, soit à composantes irréductibles lisses
et soit un diviseur à croisements normaux sur~$\sX$.
\end{lem}

\begin{proof}[Esquisse de démonstration]
Lorsque~$k$ est algébriquement clos, c'est le théorème de réduction semi-stable~\cite{kkms}.
Pour le cas général, choisissons un entier~$n$ et un modèle projectif régulier~$\sX$ de~$X$ sur~$R_n$
tel que~$Y_\red$ soit un diviseur à croisements normaux et à composantes irréductibles lisses
(cf.~\cite[Theorem~1.1]{temkin}, \cite[\textsection7.2]{dejong}).
Fixons une extension finie galoisienne $\ell/k$ telle que les composantes irréductibles
de~$Y \otimes_k \ell$ soient géométriquement irréductibles sur~$\ell$
et posons
$G=\mathrm{Gal}(\ell/k)$.
La preuve de~\cite[Theorem~4.8]{wang} fournit un multiple~$m$ de~$n$ et un
modèle $G$\nobreakdash-équivariant,
régulier et projectif~$\sX'$
 de $X \otimes_k \ell$ sur $R_m \otimes_k \ell$, dont la fibre spéciale est un diviseur réduit à croisements normaux
$G$\nobreakdash-strict au sens de \cite{wang}.  Par descente galoisienne, le schéma~$\sX'$
détermine alors un modèle de~$X$ sur~$R_m$ remplissant les conditions voulues (cf.~\cite[\textsection6.2]{BLR}).
\end{proof}

Notons~$Y^0$ l'ouvert de lissité de~$Y$ sur~$k$.  D'après le lemme de Hensel, l'indice de~$X$ sur~$K$ divise l'indice de~$Y^0$ sur~$k$.
D'autre part, comme les composantes irréductibles de~$Y$ sont lisses, l'indice de~$Y^0$ sur~$k$ est égal à l'indice de~$Y$ sur~$k$.
L'indice de~$Y$ sur~$k$ divise $\chi(Y,\sO_Y)$ puisque~$k$ vérifie la propriété~$C_1^0$ forte.
Enfin, la platitude de~$\sX$ sur~$R_n$ entraîne que $\chi(Y,\sO_Y)=\chi(X,\sO_X)$.
La proposition~\ref{prop:C10kK} est donc démontrée.
\end{proof}

Soit~$k$ une extension algébrique de $\C((x))$ de groupe de Galois absolu $\Z_2 \times \Z_3$.
Posons $K=\bigcup_{n \geq 1} k((t^{1/n}))$ et notons $k' \subset K$ la réunion des $k((t^{1/n}))$ lorsque~$n$ parcourt l'ensemble des entiers premiers à~$5$.
D'après Ax~\cite{ax}, le corps~$k'$ est de dimension cohomologique~$1$ mais n'est pas~$C_1$.
Montrons que~$k'$ vérifie néanmoins la propriété~$C_1^0$ forte.
Comme le degré de toute extension finie de~$k$ divise une puissance de~$6$,
le corps~$k$ vérifie la propriété~$C_0^0$ hors de~$6$
(cf.~remarque~\ref{rem:transition}~(i)).
Le~théorème~\ref{th:transition} et la remarque~\ref{rq:surlapreuve}~(ii)
permettent d'en déduire que~$k((t))$, et donc aussi~$k'$, vérifie la propriété~$C_1^0$ forte hors de~$6$.
D'autre part, d'après \cite[Theorem~3.1]{elw}, le corps~$k$ vérifie la propriété~$C_1^0$ forte.
Grâce à la proposition~\ref{prop:C10kK}, il s'ensuit que~$K$ vérifie la propriété~$C_1^0$ forte.
Comme~$K$ est réunion d'extensions finies de~$k'$ de degré une puissance de~$5$, on conclut que~$k'$ vérifie la propriété~$C_1^0$ forte hors de~$5$.
Ainsi, le corps~$k'$ vérifie à la fois la propriété~$C_1^0$ forte hors de~$6$ et la propriété~$C_1^0$ forte hors de~$5$:
il vérifie donc la propriété~$C_1^0$ forte.

\subsection{Corps de dimension cohomologique \texorpdfstring{$2$}{2}}

Les méthodes du présent article ne permettent pas de répondre aux quatre questions suivantes, toutes dues à Kato et Kuzumaki~\cite{katokuzumaki}.

\begin{enumerate}
\item Le corps $\Q_p$ est-il~$C_2^0$?
\item Les corps de nombres totalement imaginaires sont-ils $C_2^0$?
\item Le corps $\C((x,y))$ est-il $C_1^1$?
\item Le corps $\C(x,y)$ est-il $C_1^1$?
\end{enumerate}

La première question admet une réponse affirmative pour les hypersurfaces de degré premier (cf.~\cite{katokuzumaki}).
Les autres questions sont ouvertes même dans le cas de telles hypersurfaces.
À~tout le moins, il résulte des deux remarques ci-dessous que le corps $\C(x,y)$ ne satisfait pas la propriété~$C_1^1$ forte,
contrairement à $\C((x))((y))$ (cf.~corollaire~\ref{cor:itere}).
Nous ne savons pas ce qu'il en est pour le corps $\C((x,y))$.
Rappelons que $\C(x,y)$ et $\C((x))((y))$ sont des corps~$C_2$ (cf.~\cite{langqac}, \cite{greenberg});
c'est une question ouverte de savoir si $\C((x,y))$ est~$C_2$.

\begin{remarques}
(i) Soient $\pi:S\to \P^1_\C \times \P^1_\C$ un revêtement double ramifié le long d'une
courbe lisse très générale de bidegré~$(6,6)$ et $p:\P^1_\C \times \P^1_\C \to \P^1_\C$ la première projection.
D'après la version de Buium~\cite{buium} du théorème de Noether--Lefschetz,
l'application $\pi^*:\mathrm{NS}(\P^1_\C \times \P^1_\C) \to \mathrm{NS}(S)$ est un isomorphisme.
Par conséquent, le nombre d'intersection $(p^{-1}(0) \cdot \pi_*D)$ est pair pour tout diviseur~$D$ sur~$S$
et la fibre générique~$X$ de $p\circ\pi$ est donc d'indice~$2$.
Or~$X$ est une courbe projective, lisse, géométriquement irréductible, de genre~$2$.
Le corps~$\C(t)$ ne vérifie donc pas la propriété~$C_1^0$ forte,
bien qu'il soit~$C_1$ d'après Tsen.

(ii) Soit~$K/k$ une extension de corps.
Supposons que~$K$ soit le corps des fractions d'un anneau de valuation discrète contenant~$k$ et
de corps résiduel~$k$.
Si~$K$ vérifie la propriété~$C_1^1$ forte, alors~$k$ vérifie la propriété~$C_1^0$ forte.
En effet,
la valuation induit
une surjection $K^*/N_1(X \otimes_k K/K) \twoheadrightarrow \Z/N_0(X/k)$
pour tout $k$\nobreakdash-schéma de type fini~$X$
(cf.~\cite[\textsection2]{rost}).
\end{remarques}

\bibliographystyle{monamsalpha}
\bibliography{katokuzumaki}
\end{document}